\documentclass{amsart}

\usepackage{amssymb}

\usepackage[alphabetic]{amsrefs}

\newtheorem{theorem}{Theorem}
\newtheorem{proposition}[theorem]{Propositon}
\newtheorem{conjecture}[theorem]{Conjecture}
\newtheorem{lemma}[theorem]{Lemma}

\theoremstyle{definition}

\theoremstyle{remark}

\numberwithin{equation}{section}

\newcommand{\R}{\mathbb{R}}
\newcommand{\F}{\mathbb{F}}

\begin{document}

\title{Linear Independence of Powers for Polynomials}


\author{Alexandru Cr\u aciun}
\address{}
\curraddr{}
\email{a.craciun@tum.de}
\thanks{}


\date{}

\dedicatory{}


\begin{abstract}
	We prove the following conjecture by \cite{Kileel19}: given $ k \ge 2 $
	polynomials in $ d \ge 1 $ variables with real coefficients $ p_1, \ldots,
	p_k \in \R[x_1, \ldots, x_d] $, no two of which are linearly dependent over
	$ \R $, it holds that for any integer $ r > \left\{k \binom{k-1}{2}, 2\right\} $, their $
	r $-th powers $ p_1^r, \ldots, p_k^r $ are linearly independent.
\end{abstract}

\maketitle


\section{Introduction}
\label{section:introduction}

In \cite{Kileel19} the following result was proven:

\begin{proposition}
	\label{prop:kileel}
	Given positive integers $ d,k,s $, there exists $ \tilde{r} = \tilde{r}(d,
	k, s) $ with the following property. Whenever $ p_1, \ldots, p_k \in \R[x_1,
	\ldots, x_d] $ are $ k $ homogeneous polynomials of the same degree
	$ s $ in $ d $ variables, no two of which are linearly dependent, then $
	p_1^r, \ldots, p_k^r $ are linearly independent if $ r > \tilde{r} $.
\end{proposition}

The proof is elementary, however it is not entirely obvious how one would go
about removing the dependence of $ \tilde{r} $ on the degree $ s $ of the
polynomials. If $ \tilde{r} $ depends only on the number of polynomials and
variables considered, \cite{Kileel19} show that the interpolation ability of
neural networks with monomial activation functions could be characterized explicitly
(we refer the read to their paper for more details). Thus, they were led to
conjecture the following.

\begin{conjecture}
	\label{conj:kileel}
	In the setting of Proposition~\ref{prop:kileel}, $ \tilde{r} $ may be taken
	to depend only on $ d $ and $ k $.
\end{conjecture}

In this short paper, we show that the dependence of $ \tilde{r} $ on the degree
$ s $ of the polynomials \emph{and} on the number $ d $ of variables can be
removed. Moreover, it also turns out that the restriction to homogeneous
polynomials can be dropped as well. Before we state our result, we fix $ \F $ a field of
characteristic $ 0 $. With this, we prove the following result.

\begin{theorem}
	\label{thm:main}
	Let $ d\geq1 $ and $ k\geq2 $ be positive integers and $ p_1, \ldots, p_k
	\in \F[x_1, \ldots x_d] $ non-zero, pairwise linearly independent
	polynomials. Then the $ r $-th powers $ p_1^r, \ldots, p_k^r $ are linearly
	independent for any $ r > \max\left\{k \binom{k-1}{2},2\right\} $.
\end{theorem}

Before we present our proof, we mention that \cite{Sam19} already investigated
this question for powers of elements of integral algebras over algebraically
closed fields in arbitrary characteristic. They found out that given any $ k $
non-zero elements, there exists some integer $ 1 \leq r \leq k! $ such that the
$ r $-th powers are linearly independent over the base field. The earliest result
we have been able to find related to this problem, \cite{Bisht17}*{Theorem 6.2}, says
that for $ k $ polynomials satisfying the assumptions of Theorem~\ref{thm:main},
then their $ r $-th powers are linearly dependent for at most $ 
\binom{k-1}{2} $ values of $ r $, however no general upper bound is given. For
applications of this result to the setting of polynomial neural networks, we
refer the reader to \cite{Kileel19} which first stated the conjecture and to
\cite{Kubjas24} where a number of (counter-)examples are given.

\section{Proof of Theorem~\ref{thm:main}}
\label{section:proof}

We start by showing that it suffices to consider the univariate case with
relatively prime polynomials over an algebraically closed field. The last two
are easy to see: 1) we can always factor out the greatest common divisor $ d $
from the polynomials $ p_1, \ldots, p_k $ since they are linearly (in)dependent
if and only if their relatively prime factors $ p_1/d, \ldots, p_k/d $ are
linearly (in)dependent as well; 2) we can assume that $ \F $ is algebraically
closed since a counterexample to Theorem~\ref{thm:main} when $ \F $ is not
algebraically closed will certainly carry over when we consider the polynomials
with coefficients in the algebraic closure of $ \F $. The reduction to the
univariate case is harder, and we show how to do it in Section~\ref{sse:reduction_to_d1}. 

\subsection{Proof of Theorem~\ref{thm:main} for $ d=1 $}
\label{sse:proof_for_d1}
We will be using the generalized Mason's inequality.
\begin{proposition}[Generalized Mason's Inequality \cite{Sheppard14}*{Theorem
	1.3.1}]
	\label{prop:mason}
	Suppose $ p_1, \ldots, p_k $ are polynomials in $ \F[x] $ such that
	\begin{enumerate}
		\item
			$ p_1 + \ldots + p_k = 0 $,

		\item
			$ p_1, \ldots, p_k $ span a vector space of dimension $ k-1 $
			over $ \F $,

		\item 
			$ p_1, \ldots, p_k $ have no common zero in $ \F $.
	\end{enumerate}

	Then
	\[
		\max\{ \deg p_i \} \leq  \binom{k-1}{2}(n_0 (p_1\ldots p_k)
		- 1).
	\]
\end{proposition}
Here $ n_0 (p_1\ldots p_k) $ denotes the radical of $ p_1\ldots p_k $, i.e\ the
number of distinct roots of the polynomial. Now we can prove our main 
theorem.
\begin{theorem}
	Let $ k\geq2 $ be an integer and $ p_1, \ldots, p_k \in \F[x] $ be non-zero,
	pairwise linearly independent, relatively prime polynomials. For any $ r >
	\max\left\{ k \binom{k-1}{2}, 2\right\} $, $ p_1^r, \ldots, p_k^r $ are linearly
	independent.
\end{theorem}

\begin{proof}
	We prove the theorem by induction on $ k $. For $ k = 2 $, if one polynomial is a
	constant, the assertion follows directly. If both are non-constant, then
	assuming their $ r $-th powers are linearly dependent, we get that $ p_2^r =
	\alpha p_1^r$ for some $ \alpha \neq 0 \in \F $. Hence, they are not
	relatively prime; contradiction.

	Assume now that $ k > 2 $ and that the theorem holds for any $ 2 \leq k' < k $.
	Suppose that $ p_1^r, \ldots, p_k^r $ are linearly dependent for some $ r $.
	This means that there exist $ \beta_1, \ldots, \beta_k \in \F $ not all zero
	such that 
	\[ 
		\beta_1 p_1^r + \ldots + \beta_k p_k^r = 0. 
	\]
	We consider two cases:
	\begin{enumerate}
		\item
			$ q_1, \ldots, q_k $ span a vector space of dimension $ k-1 $, where 
			$ q_i = \beta_ip_i^r $. We can use
			Proposition~\ref{prop:mason} to bound the degree of the polynomials
			(the $ q_i $'s have no common zeroes since we assumed the $ p_i $'s
			relatively prime). Thus, we get for any $ i $
			\begin{equation*}
				\label{eq:bound}
				\deg q_i \leq  \binom{k-1}{2}  (n_0(p_1^r\ldots
				p_k^r) - 1).
			\end{equation*}
			We also have that $ n_0(p_1^r\ldots p_k^r) = n_0(p_1\ldots p_k) \leq
			\deg(p_1\ldots p_k)$. Moreover, $ \deg q_1 + \ldots + \deg q_k = r\deg
			(p_1\ldots p_k)$. Hence, summing the inequalities for all $ i $'s, we
			get
			\[ 
				r\deg (p_1\ldots p_k) \leq k  \binom{k-1}{2}  (\deg
				(p_1 \ldots p_k) - 1).
			\]
			Thus, we see that $ r <  k  \binom{k-1}{2} $.

		\item
			If some proper subset of $ p_1^r, \ldots, p_k^r $ is linearly
			dependent, we will again show that $ r < k  \binom{k-1}{2}
			$ using the induction hypothesis. Without loss of generality,
			assume that $ p_1^r, \ldots, p_n^r $ with $ 2 \leq n < k $ are
			linearly dependent. We can't directly apply the inductive hypothesis
			since $ p_1, \ldots, p_n $ may have a common factor. If they don't,
			use the induction hypothesis to show the bound on $ r $. If they do,
			denote it by $ q $. Define $ q_i = p_i/q $. The $ q_i $'s are
			relatively prime, non-zero, and pairwise independent (dividing by a
			common factor does not change linear (in)dependence). Apply the
			inductive hypothesis on the $ q_i $'s to get the bound on $ r $.
	\end{enumerate}
\end{proof}
\subsection{Reduction to $ d=1 $ case}
\label{sse:reduction_to_d1}
The idea is to find a suitable variable $ x_i $ and constants $ \alpha_1,
\ldots, \alpha_{i-1}, \alpha_{i+1}, \ldots, \alpha_d $ such that evaluating the polynomials
$ p_1, \ldots, p_k $ at this tuple leads to
non-zero, pairwise linearly independent \emph{univariate} polynomials $
\bar{p}_1, \ldots, \bar{p}_{k'} $. Then, a counterexample to
Theorem~\ref{thm:main} with $ d > 1 $, also provides a counterexample when $ d =
1 $, by substitution. We start with a lemma showing that such constants can be found
if \emph{all} polynomials $ p_1, \ldots, p_k $ share a common variable. 

\begin{lemma}
	\label{lem:projection}
	Let $ d,k>0 $ be positive integers. Given any $ p_1, \ldots, p_k
	\in \F[x_1, \ldots, x_d]$ pairwise linearly independent with $ \deg_{x_1}
	p_i > 0 $ for all $ i \in \{1, \ldots, k\} $,
	there exists a point $ (\alpha_2, \ldots, \alpha_d) \in \F^{d-1} $ such that
	$ \bar{p}_1(x_1) = p_1(x_1, \alpha_2, \ldots, \alpha_d), \ldots,
	\bar{p}_k(x_1) = p_k(x_1, \alpha_2, \ldots, \alpha_d) $ are also pairwise linearly independent.
\end{lemma}

\begin{proof}
	We look at the set of points $ (\alpha_2, \ldots, \alpha_d) \in \F^{d-1} $ such
	that some pair $ \bar{p}_i, \bar{p}_j $ is linearly dependent, and we call it $ V
	$. We show that $ V $ is a variety of dimension $ \dim V < k-1 $. Since we
	assumed $ \F $ to be of characteristic $ 0 $, it is infinite, hence
	it cannot happen that $ V = \F^{d-1} $. Thus, we can always find our desired
	point.

	Take any two polynomials $ p_i, p_j $ with $ i\neq j $. We can then write them as
	\begin{align*}
		p_i(x_1, \ldots, x_d) &= p_{i,m}(x_2, \ldots, x_d)x_1^m + \ldots +
		p_{i,0}(x_2, \ldots, x_d),\\
		p_j(x_1, \ldots, x_d) &= p_{j,n}(x_2, \ldots, x_d)x_1^n + \ldots +
		p_{j,0}(x_2, \ldots, x_d),
	\end{align*}
	and we know that $ p_{i,m} $ and $ p_{j,n} $ are not zero. Let $ (\alpha_2,
	\ldots, \alpha_k) \in \F^{k-1}$ and consider three cases:	
	\begin{enumerate}
		\item
			$ \bar{p}_i = p_i $ and $ \bar{p}_j = p_j $, whence they are linearly
		independent by assumption.
		 \item
			 $ \bar{p}_i = p_i $ but $ \bar{p}_j \neq p_j $. Then $ p_{i,l} $ are
			 constants for all $ l \in \{1, \ldots, m\} $. Assuming $ \bar{p}_i
			 $ and $ \bar{p}_j $ are linearly dependent we have that either:
			 \begin{enumerate}
			 	\item
					 $ \bar{p}_j = 0 $ in which case $ (\alpha_2, \ldots, \alpha_d) $ must
					 be a solution to the linear system of equations $ p_{j, l} = 0 $ for $ l
					 \in \{1, \ldots, n\} $, i.e.\ $ (\alpha_2, \ldots, \alpha_d) \in
					 Z(p_{j,1}, \ldots, p_{j, n}) $. The ideal $ (p_{j,1}, \ldots, p_{j,
					 n}) $ cannot be empty since at least one $ p_{j,l} $ is
					 non-constant (we assumed $ \bar{p}_j \neq p_j $).

				\item
					$ \bar{p}_j \neq 0 $ in which case $ \bar{p}_j = \beta
					\bar{p}_i $ for some non-zero $ \beta \in \F $. Then $
					(\alpha_2, \ldots, \alpha_d) \in Z(p_{j, 1} - \beta p_{i,
					1}, \ldots, p_{j, n} - \beta p_{i, n})$, which is non-empty.
			 \end{enumerate} 

		\item
			$ \bar{p}_i \neq p_i $ and $ \bar{p}_j \neq p_j	$. We have another
			two cases:
			\begin{enumerate}
				\item
					$ \bar{p}_i = 0 $ or $ \bar{p}_j = 0 $. In this case $
					(\alpha_2, \ldots, \alpha_d) \in Z(p_{i,1}, \ldots, p_{i,
					m}) \cup Z(p_{j,1}, \ldots, p_{j, n})$.
				\item 
					$ \bar{p}_j, \bar{p}_i \neq 0$. If they are linearly
					dependent, we have $ \bar{p}_j = \beta \bar{p}_i $ for some
					non-zero $ \beta \in \F $. This means that $ (\alpha_2,
					\ldots, \alpha_d) \in Z(p_{j,1} -\beta p_{i,1}, \ldots,
					p_{j,\max(n,m)} -\beta p_{i,\max(n,m)}) $.
			\end{enumerate}
	\end{enumerate}

	In either case, we see that if $ \bar{p}_i $ and $ \bar{p}_j $ are linearly
	dependent, then $ (\alpha_2, \ldots, \alpha_d) $ must lie in a variety of
	dimension strictly smaller than $ d-1 $, which we call $ V_{i,j} $. 

	Proceeding similarly for every pair of polynomials, we see that some of the
	resulting polynomials will be pairwise linearly dependent only if $
	(\alpha_2, \ldots, \alpha_d) \in \bigcup_{i < j} V_{i,j}  = V$. $ V $ is the
	finite union of varieties whose dimension is smaller than $ d-1 $, hence it
	is also a variety of dimensions strictly smaller than $ d-1 $. 
\end{proof}

Using this lemma, we can now show that Theorem~\ref{thm:main} holds for $ d>1
$ as well. Assume Theorem~\ref{thm:main} is true for $ d=1 $ and take any $ p_1, \ldots, p_k \in
\F[x_1, \ldots, x_d] $ non-zero, pairwise linearly independent polynomials such
that their $ r $-th powers are linearly dependent for some $ r > \max\left\{k
\binom{k-1}{2},2 \right\} $, i.e.\ there exist $ \beta_1, \ldots,
\beta_k \in \F $ not all zero such that 
\begin{equation}
	\label{eq:lin_dep_power}
	\beta_1p_1^r + \ldots + \beta_kp_k^r = 0.
\end{equation}

Consider the sets $ S_i = \{p_j | \deg_{x_i}p_j \neq 0\} $, i.e.\ each $ S_i $
contains the polynomials that depend non-trivially on the $ i $-th variable. If
the number of elements in each set $ |S_i| = 1 $ for all $ i $, then
Equation~\eqref{eq:lin_dep_power} cannot hold since all polynomials are
univariate, with at most one of them being a constant (because of pairwise
linear independence). Hence, there is some $ i $ such that $ |S_i| = k'> 1 $.
Without loss of generality, assume $ i=1 $. Relabel the polynomials in $ S_1 $
to $ p_1, \ldots, p_{k'} $.  Using Lemma~\ref{lem:projection} on $ p_1, \ldots,
p_{k'} $, we can find $ (\alpha_2, \ldots, \alpha_d) $ such that $ \bar{p}_1,
\ldots, \bar{p}_{k'} $ are pairwise linearly independent. If $ k' = k $,
substituting $ (\alpha_2, \ldots, \alpha_d) $ in
Equation~\ref{eq:lin_dep_power}, we get a contradiction. If $ k > k' > 1$,
evaluating $ r $-th powers polynomials \emph{not} in $ S_1 $ at  $ (\alpha_2,
\ldots, \alpha_d) $ and summing the result leads to a constant $ \gamma'\in \F
$. If $ \gamma' = 0 $, we are done. Otherwise, take $ \gamma $ equal to any $ r $-th root
of $ \gamma' $. We have that $ \bar{p}_1, \ldots, \bar{p}_{k'}, \gamma $ are
non-zero, pairwise linearly independent polynomials.  Substitution of $
(\alpha_2, \ldots, \alpha_d) $ in Equation~\ref{eq:lin_dep_power} leads, again,
to a contradiction.



\begin{bibdiv}
	\begin{biblist}
		\bib{Kileel19}{article}{
			title={On the Expressive Power of Deep Polynomial Neural Networks},
			author={Kileel, Joe},
			author={Trager, Matthew},
			author={Bruna, Joan},
			url={https://proceedings.neurips.cc/paper_files/paper/2019/file/a0dc078ca0d99b5ebb465a9f1cad54ba-Paper.pdf},
			journal={Advances in Neural Information Processing Systems},
			editor={Wallach, Hannah},
			editor={Larochelle, Hugo},
			editor={Beygelzimer, Alina},
			editor={d'Alch\'{e}-Buc, Florence},
			editor={Fox, Emily},
			editor={Garnett, Roman},
			publisher={Curran Associates, Inc.},
			volume={32},
			date={2019}

		}

		\bib{Sam19}{article}{
			title={Linear independence of powers},
			author={Sam, Steven V.},
			author={Snowden, Andrew},
			url={https://arxiv.org/abs/1907.02659},
			year={2019},
			eprint={https://arxiv.org/abs/1907.02659}
		}

		\bib{Bisht17}{thesis}{
			title={On Hitting Sets for Special Depth-4 Circuits},
			author={Bisht, Pranav},
			organization={Indian Institute of Technology Kanpur},
			type={Master's Thesis},
			date={2017},
		}

		\bib{Kubjas24}{article}{
			title={Geometry of polynomial neural networks},
			author={Kubjas, Kaie},
			author={Li, Jiayi},
			author={Wiesmann, Maximilian},
			journal={Algebraic Statistics},
			volume={15},
			number={2},
			pages={295--328},
			year={2024},
			publisher={Mathematical Sciences Publishers}
		}

		\bib{Sheppard14}{thesis}{
			title={The ABC Conjecture and its Applications},
			author={Sheppard, Joseph},
			organization={Kansas State University},
			type={Master's Thesis},
			date={2014}
		}
	\end{biblist}
\end{bibdiv}
\end{document}